\documentclass[11pt,a4paper]{amsart}
\usepackage{hyperref}
\hypersetup{colorlinks = true, linkcolor=blue}
\usepackage{amssymb, amsmath, amsthm, amscd}
\usepackage{enumerate}

\theoremstyle{plain} \newtheorem{theorem}{Theorem}[section]
\theoremstyle{plain} \newtheorem{corollary}[theorem]{Corollary}
\theoremstyle{plain} \newtheorem{proposition}[theorem]{Proposition}
\theoremstyle{plain}\newtheorem{lemma}[theorem]{Lemma}
\theoremstyle{definition} \newtheorem{definition}[theorem]{Definition}
\theoremstyle{definition}
\theoremstyle{remark}
\theoremstyle{definition}
\theoremstyle{remark}
\theoremstyle{definition}
\theoremstyle{definition}

\newcommand{\C}{{\mathbb{C}}}
\newcommand{\Z}{{\mathbb{Z}}}

\newcommand{\X}{\mathcal X}
\newcommand{\Ob}{\mathcal O}

\numberwithin{equation}{section}

\title[Automorphisms of compact complex manifolds]{A Note on the automorphism group of a compact complex manifold}
\author{Laurent Meersseman}
\date{\today}


\address{Laurent Meersseman\\ Laboratoire Angevin de REcherhe en MAth\'ematiques\\CNRS, Universit\'{e} d'Angers, Universit\'e de Bretagne-Loire\\ F-49045 Angers Cedex, France\\ laurent.meersseman@univ-angers.fr}

\begin{document}
\begin{abstract}
In this note, we give explicit examples of compact complex $3$-folds which admit automorphisms that are isotopic to the identity through $C^\infty$-diffeomorphisms but not through biholomorphisms. These automorphisms play an important role in the construction of the Teichm\"uller stack of higher dimensional manifolds.  
\end{abstract}

\maketitle

\section{Introduction.}
\label{intro}
Let $X$ be a compact complex manifold and $M$ the underlying oriented $C^\infty$ manifold. The automorphism group $\text{Aut}(X)$ of $X$ is a complex Lie group whose Lie algebra is the Lie algebra of holomorphic vector fields \cite{BM}. We denote by $\text{Aut}^0(X)$ the connected component of the identity. Its elements are thus automorphisms $f$ such that there exists a $C^\infty$-isotopy 
\begin{equation}
\label{isotopy}
t\in [0,1]\longmapsto f_t\in \text{Aut}(X)\qquad \text{ with } f_0=Id\text{ and } f_1=f.
\end{equation}
Note that $\text{Aut}(X)$ has at most a countable number of connected components so the quotient $\text{Aut}(X)/\text{Aut}^0(X)$ is discrete.

Let $\text{Diff}(M)$ be the Fr\'echet Lie group of $C^\infty$-diffeomorphisms of $M$. It is tangent at the identity to the Lie algebra of $C^\infty$ vector fields. Let $\text{Diff}^0(M)$ be the connected component of the identity. Its elements are $C^\infty$-diffeomorphisms $f$ such that there exists a $C^\infty$-isotopy 
\begin{equation}
\label{isotopy2}
t\in [0,1]\longmapsto f_t\in \text{Diff}(M)\qquad \text{ with } f_0=Id\text{ and } f_1=f.
\end{equation}
Note that the discrete group $\text{Diff}(M)/\text{Diff}^0(M)$ is the well known mapping class group.
Define now
\begin{equation}
\label{Aut1def}
\text{Aut}^1(X):=\text{Aut}(X)\cap \text{Diff}^0(M).
\end{equation}
There are obvious inclusions of groups
\begin{equation}
\label{inclusions}
\text{Aut}^0(X)\subseteq \text{Aut}^1(X)\subseteq \text{Aut}(X)
\end{equation}
In many examples, the first two groups are the same but differ from the third one (think of a complex torus). The purpose of this note is to describe an explicit family\footnote{It is indexed by two integers satisfying $b\geq 3a$ and $a>3$.} of compact $3$-folds $\X_{a,b}$ such that 
\begin{equation}
\label{property}
\text{Aut}^0(\X_{a,b})\subsetneq \text{Aut}^1(\X_{a,b})= \text{Aut}(\X_{a,b})
\end{equation}
and 
\begin{equation}
\label{mcproperty}
\text{Aut}^1(\X_{a,b})/\text{Aut}^0(\X_{a,b})=\Z_a
\end{equation}
Hence $\text{Aut}^1(\X_{a,b})$ has $a$ connected components, and this number can be chosen arbitrarily large.

Our main motivation, which is detailed in Section \ref{motive}, comes from understanding the Teichm\"uller stack of $M$, that is the quotient stack of the set of complex operators on $M$ modulo the action of $\text{Diff}^0(M)$.

The construction of the manifolds as well as the computation of their automorphism groups are elementary. As often when looking at explicit examples, the crux of the matter was to find the idea that makes everything work. We asked several specialists but they did not know any such example. We tried several classical examples but it always failed. Finally, we came accross the good family when looking for deformations of Hopf surfaces over the projective line $\mathbb P^1$ in connection with a different problem. The manifolds $\mathcal X_{a,b}$ are such deformations with the following additional property. All the fibers are biholomorphic except for those that lie above $0$ and above an $a$-th root of unity. Every automorphism must preserve these special fibers so must project onto $\mathbb P^1$ as a rotation of angle $2\pi k/a$ for some $k$. This explains the $a$ connected components of the automorphism group. Finally, diffeomorphically there is no special fiber since $\mathcal X_{a,b}$ is just a bundle. It is then not difficult to check that all rotations are allowed for diffeomorphisms so every automorphism is in $\text{Diff}^0(M)$.

\section{Motivations}
\label{motive}

Thanks to Newlander-Nirenberg Theorem \cite{NN}, a structure of a complex manifold $X$ on $M$ is equivalent to an integrable complex operator $J$ on $M$, that is a $C^\infty$ bundle operator $J$ on the tangent bundle $TM$ such that 
\begin{equation}
\label{ico}
J^2=- Id\qquad\text{ and }\qquad [T^{0,1},T^{0,1}]\subset T^{0,1}
\end{equation}
where $T^{0,1}$ is the subbundle of the complexified tangent bundle $TM\otimes\mathbb C$ formed by the eigenvectors of $J$ with eigenvalue $-i$. 

It is easy to check that the complex manifolds $X_J:=(M,J)$ and $X_{J'}$ are biholomorphic if and only if there exists a diffeomorphism $f$ of $M$ whose differential $df$ commutes with $J$ and $J'$. This defines an action of $\text{Diff}(M)$. 

The Teichm\"uller space $\mathcal T(M)$ of $M$ is then defined as the quotient of the space $\mathcal I(M)$ of integrable complex operators on $M$ (inducing the orientation of $M$) by  $\text{Diff}^0(M)$. 

As such, this is a topological space whose quotient by the mapping class group is the moduli space of complex structures on $M$. If $M$ is a surface, then this is the classical Teichm\"uller space and a complex manifold in a natural way. In higher dimensions, this is usually not even locally an analytic space, cf. \cite[Examples 11.3, 11.6]{LMStacks}. 

Kodaira-Spencer and Kuranishi classical deformation theory (see \cite{KodairaBook} for a good introduction) provides each compact complex manifold $X$ with an analytic space, its Kuranishi space, which encodes all the small deformations of its complex structure. This is not however a local moduli space, but must be thought of as the best approximation in the analytic category of a local moduli space. In particular, the same complex structure may be encoded in an infinite number of points.

In other words, there is a surjective map from the Kuranishi space $K_J$ of $X_J$ onto a neighborhood of $J$ in $\mathcal T(M)$ but which is in general far from being bijective (cf. \cite{Catsurvey} and \cite{Catsurvey2} where several results where equality holds are discussed). Especially, $\text{Aut}^0(X_J)$ acts\footnote{This is not exactly an action, cf. \cite[\S 2.3]{LMStacks} for details.} on $K_J$ identifying equivalent complex structures \cite{MeerssemanENS}, \cite[\S 2.3]{LMStacks}. The point here is that
the dimension of $\text{Aut}^0(X_J)$ is only an upper semi-continuous function of $J$ so may jump. When this occurs, the previous action is non-trivial and the Teichm\"uller space is not a local analytic space around $J$.

We are not finished yet. The quotient of $K_J$ by $\text{Aut}^0(X_J)$ is still not a priori a neighborhood of $J$ in $\mathcal T(M)$. The natural projection map may still have non-trivial discrete fibers. This phenomenon is thoroughly studied in \cite{LMStacks}. It is shown that the action of $\text{Diff}^0(M)$ onto $\mathcal I(M)$ is a foliation transversely modelled onto the stack $[K_J/\text{Aut}^0(X_J)]$ and that a holonomy groupoid can be defined for such an object. The holonomy morphisms describe exactly the non-trivial fibers (see in particular Remarks 10.4 and 10.8).  Hence, holonomy measures the gap between the stacks $\mathcal T(M)$ (locally at $J$) and $[K_J/\text{Aut}^0(X_J)]$. In many cases however, there is no holonomy. This leads to the following problem\\
\vskip1pt
\noindent {\bf Holonomy Problem.} {\sl Give examples of Teichm\"uller spaces with non-trivial holonomy.}\\
\vskip2pt
This is the main motivation of this paper.
Note that the isotropy group of a point $J$ in $\mathcal I(M)$ is exactly $\text{Aut}^1(X_J)$ and not $\text{Aut}^0(X_J)$. So letting $(M,J)$ encoding one of the manifolds $\mathcal X_{a,b}$, our result says that this isotropy group is not connected. As a consequence, the Teichm\"uller space of $M$ has non-trivial finite holonomy at $J$.

\section{The manifolds $\mathcal X_{a,b}$.}
\label{Xab}

Let $a$ and $b$ be two nonnegative integers. Let $\lambda$ be a non-zero complex number of modulus strictly less than one. For further use, we define the following two surfaces. Let $X_0$ be the Hopf surface defined as $\C^2\setminus\{(0,0)\}$ divided by the group generated by the contraction $(z,w)\mapsto (\lambda z,\lambda w)$. Let $X_1$ be the Hopf 
surface defined as $\C^2\setminus\{(0,0)\}$ divided by the group generated by the contraction $(z,w)\mapsto (\lambda z+w,\lambda w)$. These two Hopf surfaces are not biholomorphic, cf. \cite{KodairaBook}.

We consider the vector bundle $\Ob(b)\oplus\Ob(a)\to\mathbb P^1$. Throughout the article, we make use of the charts
\begin{equation}
\label{bundlecharts}
(t,z_0,w_0)\in \C^3\qquad\text{ and }\qquad(s,z_1,w_1)\in \C^3
\end{equation}
subject to the relations
\begin{equation}
\label{bundlechange}
st=1,\qquad z_1=s^bz_0,\qquad w_1=s^{a}w_0.
\end{equation}
Let $c> 0$ and let $\sigma$ be a non-zero holomorphic section of $\Ob(c)$. In accordance with \eqref{bundlecharts} and \eqref{bundlechange}, we represent it in local charts by two holomorphic maps $\sigma_0$ and $\sigma_1$ satisfying $\sigma_1(s)=s^{c}\sigma_0(t)$.
Let $W$ be $\Ob(b)\oplus\Ob(a)$ minus the zero section. 
\begin{lemma}
\label{Hopfquotient}
Assume that $b-a-c\geq 0$. Then, the holomorphic maps
\begin{equation}
\label{actionHopf1}
(t,z_0,w_0)\mapsto g_0(t,z_0,w_0)=(t,\lambda z_0+\sigma_0(t)w_0,\lambda w_0)
\end{equation}
and
\begin{equation}
\label{actionHopf2}
(s,z_1,w_1)\mapsto g_1(s,z_1,w_1)=(s,\lambda z_1+s^{b-a-c}\sigma_1(s)w_1,\lambda w_1)
\end{equation}
defines a biholomorphism $g$ of $W$.
\end{lemma}

\begin{proof}
Just compute in the other chart
\begin{equation*}
\begin{aligned}
g_1(s,z_1,w_1)=&(s,\lambda z_1+s^{b-a-c}\sigma_1(s)w_1,\lambda w_1)\\
=&(1/t,s^b(\lambda z_0+\sigma_0(t)w_0),s^a(\lambda w_0))\\
\end{aligned}
\end{equation*}
so $g_0$ and $g_1$ glue in accordance with \eqref{bundlechange}.
\end{proof}

Consider now the group $G=\langle g\rangle$. It acts freely and properly on $W$ and fixes each fiber of $W\to \mathbb P^1$. The quotient space $W/G$ is thus a complex manifold.
More precisely

\begin{proposition}
\label{Xabbundle}
The manifold $W/G$ is a deformation of Hopf surfaces over $\mathbb P^1$. Moreover the fiber over $t\in\mathbb P^1$ is biholomorphic to $X_0$ if $t$ is a zero of $\sigma$, otherwise it is biholomorphic to $X_1$.
\end{proposition}

In particular, $W/G$ is compact.

\begin{proof}
We already observed that the bundle map $W\to\mathbb P^1$ descends as a holomorphic map $\pi : W/G\to \mathbb P^1$. It is obviously a proper holomorphic
submersion, hence it defines $W/G$ as a deformation of complex manifolds parametrized by the projective line. The fiber over $t$ is $\C^2\setminus\{(0,0)\}$ divided by the contracting map $(z,w)\mapsto (\lambda z+\sigma(t)w,\lambda w)$. If $t$ is a zero of $\sigma$, then this is exactly the Hopf surface $X_0$. Otherwise, it is biholomorphic to $X_1$, see \cite{KodairaBook}.
\end{proof}

\begin{definition}
Assume that $c=2a$ and that $b\geq 3a$. We denote by $\X_{a,b}$ the manifold $W/G$ corresponding to the choice
\begin{equation}
\label{defsig}
\sigma_0(t)=t^a\prod_{k=0}^{a-1}(t-\exp (2i\pi k/a)).
\end{equation}
for $t\in\C$.
\end{definition}

For the rest of the paper, we assume that $a$ is strictly greater than $3$. Observe that the condition $b\geq 3a$ is nothing but $b-a-c\geq 0$.

\section{Computation of the automorphism groups.}
\label{autcomput}
We are in position to state and prove our main result.

\begin{theorem}
\label{main}
The manifold $\X_{a,b}$ satisfies
\begin{equation}
\label{Aut0}
 \text{\rm Aut}^0(\X_{a,b})\simeq
 \left\{ 
 \begin{pmatrix} 
 \alpha &P\\
 0 &\alpha
 \end{pmatrix} 
 \mid \alpha\in\C^*,\ P\in \C_{b-a}[X]
 \right \}/G
\end{equation}
and
\begin{equation}
\label{Aut1}
\begin{aligned}
\text{\rm Aut}(\X_{a,b})=&\text{\rm Aut}^1(\X_{a,b})\\
\simeq &\mathbb G_a\times \left\{ \begin{pmatrix} \alpha &P\\0 &\alpha\end{pmatrix} \mid \alpha\in\C^*,\ P\in \C_{b-a}[X]\right \}/G
\end{aligned}
\end{equation}
where $\mathbb G_a$ is the group of $a$-th roots of unity, $\C_{b-a}[X]$ is the space of complex polynomials with degree at most $b-a$ and where the product in {\rm \eqref{Aut1}} is given by
\begin{equation}
\label{productrule}
\left (r, \begin{pmatrix} \alpha &P\\0 &\alpha\end{pmatrix} \right )\cdot \left (r', \begin{pmatrix} \beta &Q\\0 &\beta\end{pmatrix} \right )=
\left (rr', \begin{pmatrix} \alpha\beta &\alpha Q\circ r+\beta P\circ r\\ 0 &\alpha\beta\end{pmatrix}\right )
\end{equation}
\end{theorem}

We note the immediate corollary

\begin{corollary}
\label{cor}
The group $\text{\rm Aut}^1(\X_{a,b})$ has $a> 3$ connected components and the quotient $\text{\rm Aut}^1(\X_{a,b})/ \text{\rm Aut}^0(\X_{a,b})$ is isomorphic to the cyclic group $\Z_a$.
\end{corollary}

Theorem \ref{main} will be proved through a succession of Lemmas. 

\begin{lemma}
\label{1}
Let $f$ be an automorphism of $\X_{a,b}$. Then it respects $\pi$ and descends as an automorphism $h$ of $\mathbb P^1$. 
\end{lemma}

\begin{proof}
Choose a fiber of $\X_{a,b}\to\mathbb P^1$ isomorphic to $X_1$. Restrict $f$ to it and compose with the projection onto the projective line. This gives a holomorphic map from $X_1$ to $\mathbb P^1$, hence a meromorphic function on $X_1$. But the algebraic dimension of $X_1$ is zero, see \cite{Dabrowski}, so this map is constant. In other words, $f$ sends the $\pi$-fibers isomorphic to $X_1$ onto the $\pi$-fibers. By density of these fibers, $f$ sends every $\pi$-fiber onto a $\pi$-fiber so descends  as an automorphism $h$ of $\mathbb P^1$. 
\end{proof}

\begin{lemma}
\label{2}
The automorphism $h$ is a power of the rotation at $0$ of angle $2\pi/a$.
\end{lemma}

\begin{proof}
Note that $f$ must send a $\pi$-fiber biholomorphic to $X_0$ onto a $\pi$-fiber biholomorphic to $X_0$. Now the set of such fibers is the set of $a$-th roots of unity plus zero by Proposition \ref{Xabbundle} and \eqref{defsig}. It follows from Lemma \ref{1} that the automorphism $h$ is an automorphism of the projective line which preserves this set. Since $a>3$, it must preserve at least three different points of the unit circle, hence must preserve the unit circle. But this implies that zero is fixed. Schwarz Lemma shows now that it is 
a power of the rotation\footnote{As pointed out to us by F. Bosio, this is no more true for $a=3$. Letting $j=\exp (2i\pi/3)$, the map $z\mapsto -(z-j)/(2jz+j^2)$ preserves the set but is not a rotation.} at $0$ of angle $2\pi/a$.
\end{proof}
 
 Lift $f$ as an automorphism $F$ of the universal covering $W$ of $\X_{a,b}$. We denote by $(F_0,F_1)$ its expression in the charts \eqref{bundlecharts}. 
 \begin{lemma}
\label{3}
In the charts {\rm \eqref{bundlecharts}}, the lifting $F$ has the following form
\begin{equation}
\label{f0}
F_0(t,z_0,w_0)=(r^kt, \alpha z_0+\tau_0(t)w_0,\alpha w_0)
\end{equation}
and
\begin{equation}
\label{f1}
F_1(s,z_1,w_1)=(r^{-k}s, r^{-kb}(\alpha z_1+\tau_1(s)w_1), \alpha w_1)
\end{equation}
where $r=\exp (2i\pi/a)$, $k$ is an integer, $\alpha$ a complex number and $\tau=(\tau_0,\tau_1)$ is a section of $\Ob(b-a)$.
\end{lemma}

\begin{proof}
The first coordinate in \eqref{f0} comes from Lemma \ref{2}. For the two other coordinates, recall from \cite{We} that the automorphism group of $X_0$ is $\text{GL}_2(\C)$ (modulo quotient by the group generated by the contraction) and that of $X_1$ is the group of upper triangular matrices with both entries on the diagonal equal (modulo quotient by the group generated by the contraction).

Hence the general form of $F_0$ is 
\begin{equation}
\label{f0bis}
F_0(t,z_0,w_0)=(r^kt, \alpha_0(t) z_0+\tau_0(t)w_0,\alpha_0(t) w_0)
\end{equation}
for $\alpha_0$ and $\tau$ two holomorphic functions. But in the other chart, using the same more general form of \eqref{f1}, we must have
\begin{equation*}
\begin{aligned}
F_1(s,z_1,w_1)=&(r^{-k}s, r^{-kb}(\alpha_1(s) z_1+\tau_1(s)w_1),\alpha_1(s) w_1)\\
=&(1/(r^kt), (r^{-k}s)^b(z_0\alpha_1(1/t)+\tau_1(1/t)s^{a-b}w_0), s^a\alpha_1(1/t)w_0)
\end{aligned}
\end{equation*}
which extends at $s=0$ and glues with \eqref{f0bis} if and only if $\alpha=(\alpha_0,\alpha_1)$ is a constant and $\tau=(\tau_0,\tau_1)$ is a section of $\Ob(b-a)$.

It remains to check whether these automorphisms really descend as automorphisms of $\X_{a,b}$, i.e. whether they commute with the contraction $g$ of Lemma \ref{Hopfquotient}. We compute
\begin{equation*}
\begin{aligned}
g_0\circ F_0(t,z_0,w_0)&=(r^kt,\lambda\alpha z_0+\lambda \tau_0(t)w_0+\alpha \sigma_0(r^kt)w_0,\lambda\alpha w_0)\\
&=F_0\circ g_0(t,z_0,w_0)
\end{aligned}
\end{equation*}
since $\sigma$ is $\mathbb G_a$-invariant, cf. \eqref{defsig}. A similar computation holds in the $(s,z_1,w_1)$-coordinates, so finally all these automorphisms descend.
\end{proof}

\begin{lemma}
\label{4}
An automorphism of $\X_{a,b}$ is in the connected component of the identity if and only if it descends as the identity of $\mathbb P^1$.
\end{lemma}

\begin{proof}
If an automorphism $f$ of $\X_{a,b}$ is in the connected component of the identity, then by Lemma \ref{2} its projection $h$ is isotopic to the identity through rotations of angle $2i\pi k/a$. This is only possible if $h$ is the identity. Conversely, if $h$ is the identity, it is easy to see that in Lemma \ref{3} we can move $\alpha$ to $1$ and $\tau$ to the zero section and obtain a path of automorphisms from $f$ to the identity.
\end{proof}

\begin{lemma}
\label{5}
Every automorphism of $\X_{a,b}$ is isotopic to the identity through $C^\infty$-diffeomorphisms.
\end{lemma}

\begin{proof}
Let $B_{a,b}$ be the bundle over $\mathbb P^1$ with fiber $X_0$ obtained by taking the quotient of $W$ by the group generated by the $\lambda$-homothety in the fibers. Observe that $\X_{a,b}$ can be deformed to $B_{a,b}$ putting a parameter $\epsilon\in\C$ and considering the family 
\begin{equation}
\label{family}
W\times \C/\langle \tilde g\rangle\longrightarrow \C
\end{equation}
where the action is given by (we just write it down in the first chart):
\begin{equation}
\label{actionHopfeps}
(t,z_0,w_0,\epsilon)\mapsto \tilde g_0(t,z_0,w_0,\epsilon)=(t,\lambda z_0+\epsilon\sigma_0(t)w_0,\lambda w_0,\epsilon)
\end{equation}
Hence $\X_{a,b}$ is $C^\infty$-diffeomorphic to $B_{a,b}$. More precisely, let $\X_{a,b}^\epsilon$ be the fiber of the family \eqref{family} over $\epsilon$. Then $\X_{a,b}^0=B_{a,b}$, and $\X_{a,b}^1=\X_{a,b}$ and by Ehresmann's Lemma there is an isotopy of $C^\infty$-diffeomorphisms $\phi_\epsilon$ from $\X_{a,b}^\epsilon$ onto $\X^1_{a,b}$ with $\phi_1$ equal to the identity. 

Let $f$ be an automorphism of $\X_{a,b}$. It is easy to check, using Lemma \ref{3}, that $f$ is still an automorphism of $B_{a,b}$ and of all the $\X_{a,b}^\epsilon$. We are saying that the map
\begin{equation}
\label{F}
F(t,z_0,w_0,\epsilon)=(F_\epsilon(t,z_0,w_0),\epsilon)=(f(t,z_0,w_0),\epsilon)
\end{equation}
is an automorphism of the whole family which induces $f$ on each fiber. The isotopy $\phi_\epsilon \circ F_\epsilon\circ (\phi_\epsilon)^{-1}$ joins the automorphism $f=F_1$ of $\X_{a,b}$ and $\phi_0 \circ F_0\circ (\phi_0)^{-1}$ through $C^\infty$-diffeomorphisms.

But now, at $\epsilon=0$, since $B_{a,b}$ is a holomorphic bundle, we may take any rotation at $0$ as map $h$ and thus construct a path of automorphisms $g_t$ of $B_{a,b}$ between $F_0$ and the identity. Combining the isotopy  $\phi_0 \circ g_t\circ (\phi_0)^{-1}$ with the previous one, we obtain that the automorphism $f$ of $\X_{a,b}$ is isotopic to the identity through $C^\infty$-diffeomorphisms.
\end{proof}

The proofs of Theorem \ref{main} and Corollary \ref{cor} follow easily from the previous lemmas.

\section{Final Comments.}
\label{fin}
The construction behaves well with respect to pull-backs of the base manifold $\mathbb P^1$. In particular, if $g$ is a ramified covering from $\mathbb P^1$ to $\mathbb P^1$, the pull-back manifold $\mathcal Y=g^*\mathcal X_{a,b}$ will have a more complicated finite group $\Gamma$ as quotient $\text{Aut}^1(\mathcal Y)/\text{Aut}^0(\mathcal Y)$. In this way, we may construct examples with $\Gamma$ being any cyclic or dihedral group, or with $\Gamma$ being the tetrahedral, octahedral or icosahedral group.\footnote{I owe this nice remark to David Marin.}

There are several open questions left on this topic. First the manifolds $\X_{a,b}$ are not even K\"ahler and it would be interesting to have a similar example with projective manifolds. Also, it would be interesting to have examples with $\text{Aut}^0(X)$ reduced to zero, but $\text{Aut}^1(X)$ not, especially examples of surfaces of general type, cf. \cite{Catsurvey}, \cite{Catsurvey2} and the subsequent literature. Finally, the most exciting would be to find an example with $\text{Aut}^1(X)/\text{Aut}^0(X)$ infinite since it would give a Teichm\"uller space with infinite holonomy at some point. As pointed out to us by S. Cantat, this cannot happen for K\"ahler manifolds, since the kernel of the action of $\text{Aut}(X)$ on the cohomology contains $\text{Aut}^0(X)$ as a subgroup of finite index \cite{Lieb}.  But in the non-K\"ahler world everything is possible.

For all these additional questions (except perhaps for the first one), it seems that a really different type of examples is needed.


\begin{thebibliography}{99}
\bibitem{BM} Bochner, S. and Montgomery, D.
\emph{Groups on analytic manifolds}. 
Ann. of Math. \textbf{48}, (1947), 659--669.

\bibitem{Catsurvey} Catanese, F.
\emph{A Superficial Working Guide to Deformations and Moduli}. Preprint
arXiv:1106.1368v3, to appear in the Handbook of Moduli, in honour of David Mumford, eds G. Farkas and I. Morrison, International Press.

\bibitem{Catsurvey2} Catanese, F.
\emph{Topological methods in moduli theory}. 
Bull. Math. Sci. 5, (2015). 287--449.

\bibitem{Dabrowski} Dabrowski, K.
\emph{Moduli Spaces for Hopf Surfaces}.
Math. Ann. 259, (1982). 201--225

\bibitem{KodairaBook} Kodaira, K.
\emph{Complex Manifolds and Deformations of Complex Structures}.
Springer, Berlin, 1986.

\bibitem{Lieb} Lieberman, D.I. 
\emph{Compactness of the Chow scheme: applications to automorphisms
and deformations of K\"ahler manifolds}. In Fonctions de plusieurs variables complexes, III
(S\'em. Fran\c cois Norguet, 1975-1977), pp.140--186. Springer, Berlin, 1978.

\bibitem{MeerssemanENS} Meersseman, L.
\emph{Foliated Structure of The Kuranishi Space and Isomorphism of Deformation Families of Compact Complex Manifolds}.
Ann. Sci. de l'Ecole Norm. Sup. 44, fasc. 3, (2011), 495--525.

\bibitem{LMStacks} Meersseman, L.
\emph{The Teichm\"uller and Riemann Moduli Stacks}. Preprint arXiv:1311.4170v3 (2015). 


\bibitem{NN} Newlander, A., Nirenberg, L.
\emph{Complex analytic coordinates in almost complex manifolds}.
Ann. of Math. \textbf{65}, (1957), 391--404.


\bibitem{We} Wehler, J.
\emph{Versal deformation of Hopf surfaces}.
J. Reine Angew. Math. \textbf{328}, (1981), 22--32.

\end{thebibliography}
\end{document}